\documentclass[a4paper,11pt]{article}
\usepackage[utf8]{inputenc}
\usepackage{amsfonts}
\usepackage{amssymb}
\usepackage{amsthm}
\usepackage{amsmath}
\usepackage{geometry}
\theoremstyle{plain} \newtheorem{definition}{Definition}[section]
\theoremstyle{plain} \newtheorem{Ex}{Example}[section]
\theoremstyle{plain} \newtheorem{lemma}{Lemma}[section]
\theoremstyle{plain} \newtheorem{thm}{Theorem}[section]
\theoremstyle{plain} \newtheorem{rem}{Remark}[section]
\theoremstyle{plain} \newtheorem{prop}{Proposition}[section]
\theoremstyle{plain} \newtheorem{cor}{Corollary}[section]
\theoremstyle{plain} 
\theoremstyle{plain}

\title{\textbf{Perturbation of farthest points in weakly compact sets}}
\author{J. M. AUGE}
\date{}

\begin{document}

\maketitle

\begin{abstract} If $f$ is a real valued weakly lower semi-continous function on a Banach space $X$ and $C$ a weakly compact subset of $X$, we show that the set of $x \in X$ such that $z \mapsto \|x-z\|-f(z)$ attains its supremum on $C$ is dense in $X$. We also construct a counter example showing that the set of $x \in X$ such that $z \mapsto \|x-z\|+\|z\|$ attains its supremum on $C$ is not always dense in $X$.
\end{abstract}

\begin{section}{Introduction}

Throughout this paper, $X$ denotes a real Banach space, $B_X$ its closed unit ball, $X^*$ the Banach space of all continuous linear functionals on $X$, $C$ a bounded set of $X$ and $f:X \rightarrow \mathbb{R}$ a function which is bounded below on  $C$. We study the following sets 
$$ D(C,f)=\{x \in X; \exists z \in C, r(x)=\|x-z\|-f(z)\}, $$
where by definition $r$ is the map from $X$ to $\mathbb{R}$ given by the formula 
$$ r(x)=\sup\{\|x-z\|-f(z), z \in C\}. $$
The map $r$ depends on $f$ and should be written $r_f$, but since there will be no ambiguity, we simply write $r=r_f$.
We remark that $r$ is 1-Lipschitz and convex as a supremum of such functions and that by replacing $f$ by $f+a$ where $a$ is a constant, we can suppose that $f \geqslant 0$. When $f=0$, the set $D(C,0)$ is geometrically the set of points of $X$ which admit a farthest point in the set $C$ and $r(x)$ is the farthest distance from $x$ to $C$, i. e. $r(x)$ is the smallest radius of the balls centered in $x$ that contain $C$ . Here, the function $f$ is a perturbation, we will show that under suitable hypothesis of regularity on $f$, some results known on the  set $D(C,0)$ can be generalized. To be more precise, we will be interested in the generic existence of points in $D(C,f)$. For farthest points, the problem was first studied by Edelstein in \cite{2} for uniformly convex spaces, assuming the set $C$ is bounded and norm closed and then generalized by Asplund in \cite{1} for reflexive locally uniformly convex spaces. Then Lau in \cite{4} showed that when $C$ is weakly compact (w
 ithout any geometric hypothesis on $X$), the set of farthest points is dense and he also showed that this result implies Asplund's theorem. Here we will give a generalization of Lau's theorem (see also the paper \cite{5} which deals with euclidean spaces, and \cite{3} for the case of $p$-normed spaces): when $f$ is weakly lower semi-continuous and $C$ weakly compact, the set $D(C,f)$ contains a $G_{\delta}$ dense subset of $X$. We then take some particular $f$ to see what happens when we study the set of points $x \in X$ such that $z \mapsto \|z-x\|-\|z\|$ (resp. $z \mapsto \|z-x\|+\|z\|$) attain their supremum on $C$.


\end{section}

\section{Density of the set D(C,f)}

We start this section by defining the sub-differential of the map $r$ (this definition stays unchanged for any convex map).

\begin{definition} The sub-differential of $r$ is the set
$$ \partial r(x)=\{x^* \in X^*; \forall y \in X, \langle x^*,y-x \rangle \leqslant r(y)-r(x)\}. $$
\end{definition}

Since $r$ is 1-Lipschitz, $ \partial r(x)$ is contained in the closed unit ball of the dual. We can now state our positive theorem which follows the ideas of Lau's proof.

\begin{thm}\label{thm:main}
Suppose that $C$ is a weakly compact subset of $X$ and that $f$ is weakly lower semi-continuous for the weak topology on $X$, then the set $D(C,f)$ contains a  $G_{\delta}$ dense subset of $X$.
\end{thm}

In order to prove the theorem, we will use the following lemma:
\begin{lemma} 
Let $G=\{x \in X; \forall x^* \in \partial r(x), \sup\{\langle x^*,x-z \rangle -f(z), z \in C\}=r(x)\}$. Then $G$ is a $G_{\delta}$ dense subset of $X$.
\end{lemma}

\begin{proof} Write $X \diagdown G=\bigcup_{n=1}^{\infty} F_n$ with 
$$ F_n=\{x \in X; \exists x^* \in  \partial r(x), \sup\{\langle x^*,x-z \rangle -f(z), z \in C\} \leqslant r(x)-\frac{1}{n}\}.$$
By the Baire category theorem, it is enough to show that for fixed $n \geqslant 1$, $F_n$ is closed and nowhere dense.

$\linebreak$
- Let us first show that $F_n$ is a closed subset of $X$: let $(x_k)$ be a sequence in $F_n$ converging to $x \in X$. By the definition of $F_n$, there exists $x_k^* \in \partial r(x_k)$ such that
$$ \forall z \in C, \forall k \geqslant 1, \langle x_k^*,x_k-z \rangle-f(z) \leqslant r(x_k)-\frac{1}{n}. $$ 
Since $B_{X^*}$ is compact for $\sigma(X^*,X)$, we can choose $x^* \in \bigcap_{p} \overline{\{x_k^*, k \geqslant p\}}^{\sigma(X^*,X)}$, then we get for $z \in C$: 
\begin{eqnarray*}
|\langle x_k^*,x_k-z \rangle - \langle x^*,x-z \rangle| &\leqslant& | \langle x_k^*,x_k-z \rangle - \langle x_k^*,x-z \rangle| \\ &+&|\langle x_k^*,x-z \rangle -\langle x^*,x-z \rangle| \\
&\leqslant&  \|x_k^*\| \|x_k-x\|+ |\langle x_k^*,x-z \rangle - \langle x^*,x-z \rangle| \\
&\leqslant&  \|x_k-x\|+ |\langle x_k^*,x-z \rangle - \langle x^*,x-z \rangle|.
\end{eqnarray*}
Now for each fixed $z \in C$, there exists a subsequence $(x_{k_q}^*)$ such that $\langle x_{k_q}^*,x-z \rangle$ converges, and because $x^* \in \bigcap_{p} \overline{\{x_k^*, k \geqslant p\}}^{\sigma(X^*,X)}$, this limit is $\langle x^*,x-z \rangle$. By continuity of $r$, we obtain for each $z \in C$
$$ \langle x^*,x-z \rangle -f(z) \leqslant r(x)-\frac{1}{n},$$
and hence
$$ \sup\{\langle x^*,x-z \rangle -f(z),z \in C\} \leqslant r(x)-\frac{1}{n}.$$
To conclude that $x \in F_n$, it is enough to show that $x^* \in \partial r(x)$. Indeed, since $x_k^* \in \partial r(x_k)$, we have 
$$ \forall y \in X, \langle x_k^*,y-x_k \rangle \leqslant r(y)-r(x_k) $$
so by the same argument as before, we get at the limit: $x^* \in \partial r(x)$.

$\linebreak$
- Now, let us show that each $F_n$ is nowhere dense. Suppose it is false, then one can find $y_0 \in X$ and $r>0$ such that 
$\overline{B}(y_o,r) \subset F_n$. Let $\alpha=\sup\{\|z\|, z \in C\}$, $\lambda=\frac{r}{\alpha+\|y_0\|}$ and $\varepsilon=\frac{\lambda}{n(1+\lambda)}$. By the definition of $r(y_0)$, there exists $z_0 \in C$ such that 
$$ r(y_0)-\varepsilon<\|y_0-z_0\|-f(z_0) \leqslant r(y_0). $$
Finally, put $x_0=y_0+\lambda(y_0-z_0)$.  With the choice of $\lambda$, we have $x_0 \in \overline{B}(y_0,r) \subset F_n$. Now, we estimate $r(y_0)-r(x_0)$:
$$ r(y_0)-r(x_0)<\varepsilon+\|y_0-z_0\|-f(z_0)-r(x_0). $$
But, 
$$ x_0=y_0+\lambda(y_0-z_0) \Longrightarrow x_0-z_0=(1+\lambda)(y_0-z_0). $$

Hence
\begin{eqnarray*}
r(y_0)-r(x_0) &<&\varepsilon+\frac{1}{1+\lambda}\|x_0-z_0\|-f(z_0)-r(x_0) \\ 
&=&\varepsilon+\frac{1}{1+\lambda}(\|x_0-z_0\|-f(z_0))+(\frac{1}{1+\lambda}-1)f(z_0)-r(x_0) \\ 
&\leqslant& \varepsilon+\frac{1}{1+\lambda}r(x_0)-\frac{\lambda}{1+\lambda}f(z_0)-r(x_0) \\
&=&\varepsilon-\frac{\lambda}{1+\lambda}r(x_0)-\frac{\lambda}{1+\lambda}f(z_0). 
\end{eqnarray*}

Since $x_0 \in F_n$, there exists $x^* \in \partial r(x_0)$ such that   
$$ r(x_0) \geqslant \sup \{\langle x^*,x_0-z \rangle -f(z), z \in C\}+\frac{1}{n} \geqslant \langle x^*,x_0-z_0 \rangle-f(z_0)+\frac{1}{n}, $$
which gives, combined with the last estimation:
$$ r(y_0)-r(x_0)<\varepsilon-\frac{\lambda}{1+\lambda}\langle x^*,x_0-z_0 \rangle-\varepsilon
=\langle x^*,y_0-x_0 \rangle,$$
which contradicts $x^* \in \partial r(x_0)$.
\end{proof}

Here, we have just used the fact that $C$ is bounded. The hypothesis of weak compactness of $C$ and of weak lower semi-continuity of $f$ allow us to finish the proof of the theorem as follows 

\begin{proof} It is enough to see that $G \subset D(C,f)$. Consider $x \in G$ and $x^* \in \partial r(x)$, so 
$$ \sup\{\langle x^*,x-z \rangle-f(z), z \in C\}=r(x). $$
Since $f$ is weakly lower semi-continuous and that $z \mapsto \langle x^*,x-z \rangle$ is weakly continuous, then $z \mapsto \langle x^*,x-z \rangle-f(z)$ is weakly upper semi-continuous on the weakly compact set $C$, and attains its supremum at a point $z_0$. We get:
$$ r(x) \leqslant \|x^*\|\|x-z_0\|-f(z_0) \leqslant r(x) $$ because $\|x^*\| \leqslant 1$ and hence $r(x)=\|x-z_0\|-f(z_0)$.
\end{proof}

Since $z \mapsto \|z\|$ is weakly lower semi-continuous, we obtain

\begin{cor}\label{cor:main} If $C$ is weakly compact, the set of $x \in X$ such that $z \mapsto \|x-z\|-\|z\|$ attains its supremum on $C$ is dense in $X$.
\end{cor}

\section{Counter examples and remarks} 

It is natural to ask ourselves if we can drop the hypothesis of weak lower semi-continuity in Theorem \ref{thm:main}. The answer is no: more precisely, we construct the following counter example 

\begin{Ex} If $(K,d)$ is an infinite compact metric space and if $X=C(K)$ is the space of real continuous functions on K equiped with its usual norm, there exists a weakly compact subset $C$ of $X$ and a function $f$  weakly upper semi-continuous on $X$ such that $D(C,f)$ is not dense in $X$.
\end{Ex} 

Indeed, take $f(z)=(1-\|z\|)^{+}=\max(0,1-\|z\|)$ and consider a decreasing sequence $(U_n)_{n \geqslant  1}$ of open subsets of $K$ such that $\bigcap_{n \geqslant 1} U_n=\emptyset$ (fix $y \in K$ which is not an isolated point in $K$, then a possible choice is $U_n=\{x \in K \setminus \{y\}; d(x,y)<\frac{1}{n}\}$), let us also fix $t_n \in U_n$ and put
$$ x_n(t)=\frac{d(t,{U_n}^c)}{d(t,t_n)+d(t,{U_n}^c)} \; \; (t \in K, n \geqslant 1). $$
By construction of $U_n$, we have $\|x_n\|=1$ and $(x_n)_{n \geqslant 1}$ converges pointwise to $0$ which implies that $(x_n)_{n \geqslant 1}$ converges weakly to $0$ as easily seen using the Riesz representation theorem and the Lebesgue's dominated convergence theorem. Put
$$ C=\{(1-\frac{1}{n})x_n, n \geqslant 1\}=\{0\} \cup \{(1-\frac{1}{n})x_n, n \geqslant 2\} $$
which is weakly compact as the union of a convergent sequence and its limit. Note that $C$ is contained in $B_X$ and hence $f(z)=1-\|z\|$, we are left to find the supremum of the function $f_x$ ($x \in X$ fixed) defined for $z \in C$ by $f_x(z)=\|x-z\|+\|z\|$. We will show that for $x \in \overline{B}(\textbf{2},1)$ (where $\textbf{2}$ denotes the function identically equal to 2), $f_x$ never attains its supremum and as a consequence $D(C,f)$ is not dense. Since for $t \in K$, $x(t) \geqslant 1$, we get for $z \in C$
$$ \|x-z\|=\sup|x(t)-z(t)|=\sup(x(t)-z(t)) \leqslant \sup x(t)=\|x\| $$
and on the other hand $\|z\|<1$ gives $f_x(z)<\|x\|+1$. To finish, the last thing we have to see is that $\sup f_x \geqslant \|x\|+1$. Fix $t_0$ such that $\|x\|=|x(t_0)|$, then 
$$ \sup f_x \geqslant f_x((1-\frac{1}{n})x_n)) \geqslant |x(t_0)-(1-\frac{1}{n})x_n(t_0)|+(1-\frac{1}{n}). $$
The conclusion follows because $(x_n)_{n \geqslant 1}$ converges pointwise to $0$.

\begin{rem} 
- This last example also shows that the set of $x \in X$ such that $z \mapsto \|z-x\|+\|z\|$ attains its supremum on $C$ is not always dense in $X$. Recall that according to Corollary \ref{cor:main}, the set of $x \in X$ such that $z \mapsto \|z-x\|-\|z\|$ attains its supremum on $C$ is always dense in $X$. \\
- There exists spaces, for example $l^1(\mathbb{N})$, or more generally any Banach space with the Schur's property where we can't construct any counter examples of the above type because the weakly and strongly compact sets coincide. \\
- However if $C=B_X$ and $X$ is reflexive (to ensure the weak compactness of $C$). The set of $x$ such that $f_x$ (defined by $f_x(z)=\|x-z\|+\|z\|$) attains its supremum on $C$ is dense. To show this, we use the following proposition.
\end{rem}

\begin{prop} Let $f$ be a continuous convex function on $X$, $C$ a weakly compact subset of $X$ and $\varepsilon(C)$ the set of extremal points of $C$, then $\sup_{C} f=\sup _{\varepsilon(C)} f$.

\end{prop}

\begin{proof} We have obviously, $\sup_{\varepsilon(C)} f \leqslant \sup _{C} f$. Suppose the reverse inequality is false and introduce $t$ such that 
$$ \sup_{\varepsilon(C)} f < t < \sup _{C} f.$$
Then, we have $\varepsilon(C) \subset C_0:=\{f \leqslant t\}$. Since $f$ is continuous convex , $C_0$ is a closed convex set, the Krein-Milman's theorem says that $\overline{\textrm{conv}}^{\| . \|}(\varepsilon(C))=C$, hence $C \subset C_0$. Now, since  $\sup _{C} f>t$, one can find $x \in C$ such that $f(x)>t$ which contradicts $x \in C_0$.
\end{proof}

This implies the last remark, indeed $\varepsilon(C)$ is of course contained in the unit sphere. Using the previous fact two times, we see that
$$ \sup_{z \in C} f_x(z)=\sup_{z \in \varepsilon(C)} f_x(z)=1+\sup_{z \in \varepsilon(C)} \|x-z\|=1+\sup_{z \in C} \|x-z\| $$
which gives the conclusion with the main theorem (with the pertubation $f=0$).$
\linebreak$

\begin{rem}
To finish, we would like to mention that the map $f \mapsto D(C,f)$ has no good properties. Let us take  $X=\mathbb{R}$, $C=[0,1]$ and put for $z \in \mathbb{R}$, $f_k(z)=\frac{\textbf{1}_{\{0,1\}}(z)}{k}$ where $\textbf{1}_{\{0,1\}}$ denotes the characteristic function of the pair $\{0,1\}$ which is equal to 1 if $z=0$ or $z=1$ and $0$ otherwise. It is obvious that $(f_k)_{k \geqslant 1}$ converges uniformly to $0$ ($D(C,0)=X$) and yet, all the $D(C,f_k)$ are empty. 
\end{rem}
Indeed, let $x \in \mathbb{R}$ and suppose that $x \geqslant \frac{1}{2}$. For $z \in [0,1]$, $|x-z|$ is maximal when $z=0$ and is equal to $x$. Hence 
$$ \sup \{|x-z|-f_k(z), z \in [0,1]\} \leqslant x. $$
On the other hand, taking a sequence $(z_n) \subset ]0,1[$ converging to $0$, we get the reverse inequality. If we had a $z$ which attains the supremum, we should have 
$$ f_k(z)=|x-z|-x \leqslant x-x=0,$$
which implies that $z \in ]0,1[$. This gives us $|z-x|=x$ with $z \in ]0,1[$, which contradicts $|x-z|<x$. For $x \leqslant \frac{1}{2}$, we proceed the same way with the point $z=1$. \\ 

\textbf{Acknowledgments.} I would like to thank my advisor Robert Deville for his help during the elaboration of this paper as well as the referee for useful comments.

\small{
}


\begin{thebibliography}{}
\bibitem[1]{1} E. Asplund, Farthest points in reflexive locally uniformly rotund Banach spaces, Israel J. Math. 4 (1966), 213-216.
\bibitem[2]{2} M. Edelstein, Farthest points of sets in uniformly convex Banach spaces, Israel J. Math. 4 (1966), 171-176.
\bibitem[3]{3} S. Hejazian, A. Niknam, and S. Shadkam, Farthest Points and Subdifferential in
p-Normed Spaces, Hindawi Publishing Corporation
International Journal of Mathematics and Mathematical Sciences,
Volume 2008, Article ID 196326, 6 pages.
\bibitem[4]{4} K.S. Lau, Farthest points in weakly compact sets, Israel J. Math. 22 (1975), 168-174.
\bibitem[5]{5} X. Wang, On Chebyshev Functions and Klee Functions, submitted to Journal of Mathematical Analysis and Applications. 
\end{thebibliography}
\end{document}